%% file: ex_article.tex
\documentclass[final]{siamart250211}

\input{ex_shared}

\usepackage{enumerate,amsmath,amssymb,graphicx,subcaption}
\newtheorem{example}{Example}
\newtheorem{assumption}{Assumption}
\makeatletter
\newsavebox{\@brx}
\newcommand{\llangle}[1][]{\savebox{\@brx}{\(\m@th{#1\langle}\)}%
  \mathopen{\copy\@brx\kern-0.5\wd\@brx\usebox{\@brx}}}
\newcommand{\rrangle}[1][]{\savebox{\@brx}{\(\m@th{#1\rangle}\)}%
 \mathclose{\copy\@brx\kern-0.5\wd\@brx\usebox{\@brx}}}

\ifpdf
\hypersetup{
  pdftitle={An Example Article},
  pdfauthor={D. Doe, P. T. Frank, and J. E. Smith}
}
\fi

\begin{document}

\maketitle

\begin{abstract}
In this paper,  we  consider  the  integrating factor midpoint method  for wave-type equations
 and derive optimal order a posteriori error estimates. We first introduce  an
  integrating factor midpoint approximation defined  by the piecewise linear approximate solutions,
   and derive suboptimal order residual-based error estimates by employing the energy technique.
  Hence the key is to introduce a continuous, piecewise quadratic time reconstruction to establish
   optimal order error bounds. Based on the reliable  a posteriori error control,
   we develop an adaptive time-stepping strategy. Numerical examples are implemented to
    verify the convergence rate of the a posteriori error estimator and the high efficiency of
    the adaptive algorithm.
\end{abstract}



\section{Introduction}
A posteriori error estimates utilize discrete solutions and the data of problems to provide
 computable error bounds, thereby facilitating adaptive computations. Adjusting the mesh and
  time step-size  is  essential for  optimizing  computational costs.  Since Babu\v{s}ka and
  Rheinboldt \cite{Babuska1978} pioneered  adaptive finite element methods through  rigorous
  mathematical derivations,  a posteriori error analysis  has garnered significant attention
  for both elliptic problems and evolution equations discretized in space or
  time (see, e.g., \cite{Bank1985,Verfurth1994,Eriksson1995,Ainsworth1997,Nochetto2000,Makridakis2003,Chen2004,Lakkis2006}). Particularly over the past two decades,  research has intensified on a posteriori error estimates and adaptivity of time-stepping methods applied to evolution equations.
  For instance, reconstruction techniques have been used to derive optimal order
   a posteriori error bounds  for  classical (non-exponential) numerical
    methods (see, e.g., \cite{Akrivis2006,Akrivis2010,Akrivis2011}) for parabolic equations.
     In \cite{Ern2015,Ern2017}, a posteriori error analysis for discontinuous
     Galerkin methods for parabolic problems was established through an
     equilibrated flux approach. Optimal order a posteriori error estimates
      for finite element methods,  leap-frog and cosine methods, and Runge--Kutta
       discontinuous Galerkin methods for abstract  second order evolution problems
       were derived in \cite{Huang2013,Georgoulis2016,Georgoulis2024}.
       The reconstruction technique has also been extended  to delay
        differential equations \cite{Wang2018,Wang2020,Wang2022b}, and adaptive
        algorithms were  further developed  for parabolic and fractional parabolic
         problems \cite{Wang2021,Wang2022a,Cao2025}; the latter specifically addresses
          the nonlocality and singularity of fractional derivatives.

  Exponential integrators have been developed and applied in various  fields due to their favorable
   convergence properties and stability (see, e.g.,
   \cite{Cox2002,Hochbruck2005,Kassam2005,Calvo2006,Celledoni2009,Celledoni2011,Bao2014,Li2016,Wang2016,Ostermann2018}).
   Hochbruck and Ostermann \cite{Hochbruck2010}  presented  a nice review about the construction and
   analysis of exponential integrators, and for methods to compute the products of matrix
    exponentials and vectors, we refer to \cite{Moler2003,Higham2008,Al-Mohy2011}.  In  recent years,
     maximum bound principle preserving  exponential   time-differencing schemes (ETD) \cite{Du2021},
       exponential cut-off methods \cite{Li2020}, and integrating factor Runge--Kutta methods
       \cite{Ju2021} were formulated for the Allen--Cahn equation.  Symmetric exponential-type
       integrators were constructed for highly oscillatory Hamiltonian systems (see, e.g., \cite{Wang2023,Bao2024}).
       We note that the variable step-size ETD method has been applied to  semilinear problems through
        error estimates of the interpolated nonlinear term \cite{Calvo2008}, and adaptive Krylov subspace methods
         for exponential integrators were considered in \cite{Niesen2012,Gaudreault2018,Bergermann2024}.
         However, there has been limited literature devoted to the a posteriori error analysis for  exponential integrators
          for evolution equations,  with the exception of  our recent work \cite{Hu2025a,Hu2025b}, which exclusively addresses
          first order differential equations.

The integrating factor  (IF) method,  also known as the Lawson-type method, introduces an appropriate
 change of variables to the original system, and then applying a numerical scheme to the transformed
 system.  Owing to its conceptual simplicity and straightforward implementation, IF methods have
  sustained considerable attention (see, e.g., 
  \cite{Lawson1967,Krogstad2005,Mei2017,Isherwood2018,Ostermann2020,Li2021,Feng2024}).
   Classical (non-stiff) order conditions of IF Runge--Kutta methods for first order differential
    systems were analyzed in \cite{Berland2005}, and  the convergence of Lawson-type methods
    was considered in \cite{Hochbruck2020}.  Motivated by \cite{Akrivis2006} and \cite{Georgoulis2016},
    this work is concerned with the  a posteriori error  analysis for the IF
    midpoint method for the following linear second order hyperbolic equation
\begin{equation}\label{linearhyperbolicequation}
\left\{
\begin{aligned}
u^{\prime\prime}(t)+Au(t)&=f(t),\quad 0\leq t\leq T,\cr\noalign{\vskip1.5truemm}
 u(0)&=u^0,\cr\noalign{\vskip1.5truemm}
  u^{\prime}(0)&=v^0,\cr\noalign{\vskip1.5truemm}
\end{aligned}
\right.
\end{equation}
where  $A$ is a positive definite, self-adjoint, and linear operator on a  Hilbert space
  $(H,\langle \cdot,\cdot \rangle)$ with  $D(A)$ dense in $H$,  $f:[0,T]\rightarrow H$,
  and  $u^0$, $v^{0}\in H$.  For the second order method,   we begin by introducing  the continuous
    piecewise linear  approximation $(U,V)$ to $(u,u^{\prime})$,  and derive suboptimal error bounds
     in the $L^{\infty}$-in-time and energy-in-space norms for  $u-U$ and $u^{\prime}-V$.
     To recover the optimal order, we then  introduce a piecewise quadratic time
     reconstruction $(\hat{U},\hat{V})$, and establish optimal order  residual-based error estimates
     for $u-\hat{U}$ and $u^{\prime}-\hat{V}$. It should be emphasized that error bounds derived in
      this paper are  explicitly computable and do not rely on specialized techniques for evaluating
      the products of matrix exponentials and vectors, in contrast to the methods proposed
      in \cite{Niesen2012,Gaudreault2018,Bergermann2024}. Furthermore, the adaptive time-stepping
      scheme is constructed using the reliable a posteriori error estimate. Numerical results
      demonstrate the reliability of the a posteriori error estimator and the effectiveness of our adaptive
       algorithm.

The rest of this paper is organized as follows. In Section $2$,
we introduce the IF midpoint approximation and conduct the suboptimal
 a posteriori error analysis. Section $3$ is devoted to a piecewise quadratic time reconstruction
 of this approximation, and derives optimal order a posteriori error bounds. In Section $4$,
 we develop an adaptive time-stepping method. Several numerical examples are carried out  to
 verify the theoretical results in  Section $5$. The final section provides  concluding remarks.

\section{IF midpoint method for linear problems}
Denoting $v:=u^{\prime}$, the linear hyperbolic equation \eqref{linearhyperbolicequation}  can be written as the following first order system
\begin{equation}\label{first-order-system}
\begin{aligned}& \left(
                   \begin{array}{c}
                     u(t)\\
                      v(t)
                   \end{array}
                 \right)
^{\prime}+ \left(
    \begin{array}{cc}
      0& -I\\
       A&0
    \end{array}
  \right)\left(
                   \begin{array}{c}
                    u(t)\\
                      v(t)
                   \end{array}
                 \right)=
\left(                                                         \begin{array}{c}                                             0\\
f(t)                               \end{array}
 \right).
\end{aligned}
\end{equation}
Let us denote $z=(u(t),v(t))^{\intercal}$. Then the system  \eqref{first-order-system} can be simplified into
\begin{equation}\label{linearproblem}
  z^{\prime}(t)+Mz(t)=\bar{f}(t),
\end{equation}
with $\bar{f}(t)=(0,f(t))^{\intercal}$, and
\[M=
\left(
\begin{array}
[c]{cc}
    0& -I\\
       A&0\\
\end{array}
\right).
\]
The domain of the linear operator $M$ is given by $D(M)=D(A)\times D(A^{1/2})$. We define $X=D(A^{1/2})\times H$ with the norm $\|\cdot\|_{X}$, and $\|\cdot \|_{X\leftarrow X}$ denotes the operator norm on $X$.
Our analysis is based on the following classical assumption, we refer to \cite{Evans1998,Hochbruck2010}.
\begin{assumption}\label{assum1}
The linear operator $-M$ generates a $C_0$-semigroup $\{e^{-tM}\}_{t \geq 0}$ on $X$, and there exists a constant $C>0$ such that
\begin{equation*}
\|e^{-tM}\|_{X\leftarrow X}  \leq C, \quad \forall t\geq0.
\end{equation*}
\end{assumption}

\subsection{IF midpoint approximation}\label{sec:IFMD for hyperbolic equations}
Let us consider a partition $0=t^{0}<t^{1}<\cdots <t^{N}=T$ of the interval $[0,T]$ with $J_{n}:=(t^{n-1},t^{n}]$, and $k_{n}:=t^{n}-t^{n-1}$. For given $\{v^{n}\}_{n=0}^N$, we will use the following notation
\begin{equation*}
\bar{\partial} v^n:=\frac{v^n-v^{n-1}}{k_n},\quad  v^{n-\frac{1}{2}}:=\frac{v^n+v^{n-1}}{2}, \quad  n=1,\ldots,N,
\end{equation*}
Applying the second-order IF method (see \cite{Celledoni2008}, Example $3.2$) to the system \eqref{first-order-system}, the nodal approximations $Z^{n}\in D(M)$ to the values $z(t^n)$ are obtained by
\begin{equation}\label{IFmethod}
  Z^{n}=e^{-k_nM}Z^{n-1}+k_ne^{-\frac{1}{2}k_nM}\bar{f}(t^{n-\frac{1}{2}}), \quad n=1,\ldots,N,
\end{equation}
with $Z^{0}=(U^{0},V^{0})^{\intercal}\in D(M)$ approximating $(u^{0},v^{0})^{\intercal}$, which implies that
\begin{equation*}
\begin{aligned}& \left(
                   \begin{array}{c}
                    U^{n}\\
                      V^{n}
                   \end{array}
                 \right)
= e^{-k_nM}\left(
    \begin{array}{c}
  U^{n-1}\\
  V^{n-1}
    \end{array}
  \right)+k_ne^{-\frac{1}{2}k_nM}
\left(                                                \begin{array}{c}                                             0\\
 f(t^{n-\frac{1}{2}})          \end{array}
 \right),  \quad n=1,\ldots,N.
\end{aligned}\end{equation*}
It is simple to verify that the method  \eqref{IFmethod} is symmetric by the changes $k_n\leftrightarrow -k_n$ and $n\leftrightarrow n-1$.   Celledoni, Cohen, and Owren \cite{Celledoni2008} formulated symmetric exponential integrators for the cubic Schr\"{o}dinger equation, and illustrated that
the IF midpoint method preserves the symmetry  and $L^2$-norm of the system. To better facilitate  theoretical analysis,
 the $\varphi$-functions  are defined by the entire functions
 \begin{equation*}
\varphi_l(z)=\int_{0}^{1} e^{(1-\tau)z}\frac{\tau^{l-1}}{(l-1)!}d\tau, \quad l\in \mathbb{N}^{+},
\end{equation*}
and satisfy the recurrence relation
\begin{equation*}
  \varphi_{l+1}(z)=\frac{\varphi_{l}(z)-\varphi_l(0)}{z},\quad \varphi_{0}(z)=e^{z}, \quad \varphi_{l}(0)=\frac{1}{l!}.
\end{equation*}
Using the relation  $e^{-k_nM}=I-k_n\varphi_1(-k_nM)M$, we have
\begin{equation*}
 Z^{n}=Z^{n-1}-k_n\varphi_1(-k_nM)MZ^{n-1}+k_ne^{-\frac{1}{2}k_nM}\bar{f}(t^{n-\frac{1}{2}}), \quad n=1,\ldots, N.
 \end{equation*}
Hence
 \begin{equation}\label{IFMM1}
\bar{\partial}Z^{n}=-\varphi_1(-k_nM)MZ^{n-1}+e^{-\frac{1}{2}k_nM}\bar{f}(t^{n-\frac{1}{2}}), \quad n=1,\ldots,N.
 \end{equation}

For the IF midpoint method with order two,  it is natural to introduce the  approximation
$Z=(U,V)^{\intercal}:[0,T] \rightarrow D(M)$ to $z$, defined by  linearly   interpolating between  the nodal values $Z^{n-1}$ and $Z^{n}$,
\begin{equation}\label{linearinterpolation}
  Z(t)=Z^{n-1}+(t-t^{n-1})\bar{\partial}Z^{n}, \quad t\in J_{n},
\end{equation}
which means that
\begin{equation*}
\begin{aligned}& \left(
                   \begin{array}{c}
                    U(t)\\
                      V(t)
                   \end{array}
                 \right)
= \left(
    \begin{array}{c}
  U^{n-1}\\
  V^{n-1}
    \end{array}
  \right)+(t-t^{n-1})
\left(                                                \begin{array}{c}                                             \bar{\partial} U^{n}\\
\bar{\partial} V^{n}          \end{array}
 \right), \quad t\in J_{n}.
\end{aligned}\end{equation*}
Let $R(t)=(R_{U},R_{V})^{\intercal}\in X$ be  the residual of the approximation $Z$,
\begin{equation}\label{initial residual}
  R(t):=Z^{\prime}(t)+MZ(t)-\bar{f}(t), \quad t \in J_{n},
\end{equation}
which can be considered as  the amount by which the approximation solution $Z$ misses
satisfying \eqref{first-order-system}. In view of \eqref{IFMM1} and \eqref{linearinterpolation}, the residual $R(t)$ can be written as
\begin{equation}\label{linear residual}
\begin{aligned}
  R(t)&=\bar{\partial}Z^{n}+M\big(Z^{n-1}+(t-t^{n-1})\bar{\partial} Z^{n}\big) -\bar{f}(t)\cr\noalign{\vskip2truemm}
  &=\big(I-\varphi_1(-k_nM)\big)MZ^{n-1}+(t-t^{n-1})M\bar{\partial}Z^{n}+(e^{-\frac{1}{2}k_nM}-I)\bar{f}(t^{n-\frac{1}{2}}) \cr\noalign{\vskip2truemm}
  &\quad +\bar{f}(t^{n-\frac{1}{2}})-\bar{f}(t), \quad t\in J_{n},
\end{aligned}
\end{equation}
i.e.,
\begin{equation*}
\begin{aligned}
\left(
                   \begin{array}{c}
                    R_U(t)\\
                    R_V(t)
                   \end{array}
                 \right)
&= (I-\varphi_1(-k_nM))M\left(
    \begin{array}{c}
  U^{n-1}\\
  V^{n-1}
    \end{array}
  \right)+(t-t^{n-1})M
\left(                                                \begin{array}{c}                                             \bar{\partial} U^{n}\\
\bar{\partial} V^{n}          \end{array}
 \right) \cr\noalign{\vskip1.5truemm}
 &\quad +(e^{-\frac{1}{2}k_nM}-I)
 \left(                                                \begin{array}{c}                                             0\\
f(t^{n-\frac{1}{2}})        \end{array}
 \right)+
 \left(                                                \begin{array}{c}                                             0\\
f(t^{n-\frac{1}{2}}) -f(t)      \end{array}
 \right)
 .
\end{aligned}\end{equation*}

Under Assumption \ref{assum1}, we define the following graph norm
\begin{equation*}
  \|z\|_{D(M)}:=\|z\|_{X}+\|Mz\|_{X}, \quad \forall z \in D(M),
\end{equation*}
and the following estimate holds for all $z\in D(M)$,
\begin{equation*}
\begin{aligned}
  \|(e^{-tM}-I)z\|_{X}&=\Big\|-\int_{0}^{t} e^{-sM}Mzds\Big\|_{X} \leq \int_{0}^{t} \| e^{-sM}Mz\|_{X} ds \cr\noalign{\vskip2truemm}
  &\leq Ct\|Mz\|_X.
  \end{aligned}
\end{equation*}
Hence, we have
\begin{equation*}
\begin{aligned}
  \|e^{-tM}-I\|_{X\leftarrow D(M)}:&=\sup_{\substack{z \in D(M) \\ z \neq 0}} \frac{\|(e^{-tM}-I)z\|_{X}}{\|z\|_{D(M)}} \leq \sup_{\substack{z \in D(M) \\ z \neq 0}}  Ct
  \frac{\|Mz\|_X}{\|z\|_{D(M)}}\leq Ct.
  \end{aligned}
\end{equation*}
Based on the operator norm estimates previously established, it is straightforward to demonstrate that
\begin{equation}\label{exponential property1}
 \|e^{-\frac{1}{2}k_nM}-I\|_{X\leftarrow D(M)}=\mathcal{O}(k_n).
\end{equation}
Through similar derivations leveraging the integral representation, we have the corresponding estimate for the $\varphi_1$-function:
\begin{equation}\label{exponential property2}
\|I-\varphi_1(-k_nM)\|_{X\leftarrow D(M)}=\mathcal{O}(k_n).
\end{equation}
From \eqref{exponential property1} and \eqref{exponential property2}, it follows that the residual $R(t)$ is of suboptimal (first) order.

\subsection{Suboptimal a posteriori error analysis}\label{sec:suboptimal error estimates}
For $\Phi=(\phi_1,\phi_2)$, ${E}=(E_1,E_2)$ $ \in D(A^{1/2}) \times H$, let us define
the bilinear form
\begin{equation*}
  \llangle \Phi^{\intercal},\mathcal{E}^{\intercal} \rrangle := \langle A^{1/2}\phi_1,A^{1/2}E_1 \rangle+\langle \phi_2,E_2\rangle.
\end{equation*}
 Obviously, $\llangle \cdot,\cdot \rrangle$ defines an energy inner product on $[D(A^{1/2})\times H]^2$, and its induced energy norm is denoted by  $|\Vert \cdot \Vert|$,
\begin{equation*}
  |\Vert \Phi \Vert|=\big(\Vert A^{1/2} \phi_1 \Vert^2+\Vert \phi_2 \Vert^2\big)^{1/2}.
\end{equation*}

\begin{theorem}\label{suboptimal error bound}
Let $u$ be the exact solution of \eqref{linearhyperbolicequation}, $z=(u,v)^{\intercal}$ satisfies \eqref{linearproblem}, and  $Z=(U,V)^{\intercal}$ be the approximation to $z$ defined by \eqref{linearinterpolation}.  Denoting the error by $e:=z-Z=(e_U,e_V)^{\intercal}$, for  any   $0<\theta<1/2$, the following a posteriori error estimate holds:
\begin{equation*}
\begin{aligned}
 \displaystyle & \max_{t\in[0,t^{N}]}|\Vert (e_U,e_V)(t)\Vert|^2\cr\noalign{\vskip0.2truemm}
  \leq  & \ \frac{1}{1-2\theta}|\Vert (e_U,e_V)(0)\Vert|^2 +
 \frac{1}{2\theta-4\theta^2} \Big( \int_{0}^{t^{N}} |\Vert ({R}_U,R_{V}) (s)\Vert|ds\Big)^2.
 \end{aligned}
\end{equation*}
\end{theorem}
\begin{proof}
Subtracting  \eqref{initial residual} from \eqref{linearproblem}, we have the following error equation
\begin{equation*}
e^{\prime}(t)+Me(t)=-R(t),
\end{equation*}
and an easy computation gives
\begin{equation}\label{linearerrorequation}
\begin{aligned}
\left(    \begin{array}{c}
                    {e}_U\\
                     {e}_V
                   \end{array}
                 \right)^{\prime}
                 +
     \left(    \begin{array}{c}
                   - {e}_V\\
                     Ae_U
                   \end{array}
                 \right)
=-
\left(    \begin{array}{c}
                    R_U\\
                    R_V
                   \end{array}
                 \right).
\end{aligned}\end{equation}
We take the inner product of  \eqref{linearerrorequation} with $(e_{U},e_{V})^{\intercal}$, and get
\begin{equation*}
\begin{aligned}
   \displaystyle \llangle ({e}_U^{\prime},{e}_V^{\prime})^{\intercal},({e}_U,{e}_V)^{\intercal} \rrangle =
   \llangle (e_{V},-Ae_U)^{\intercal},(e_U,e_V)^{\intercal} \rrangle +
 \llangle -({R}_U,R_{V})^{\intercal},({e}_U,{e}_V)^{\intercal} \rrangle.
 \end{aligned}
\end{equation*}
Using the  self-adjointness of $A$, we arrive at
\begin{equation*}
\begin{aligned}
   \displaystyle \llangle ({e}_U^{\prime},{e}_V^{\prime})^{\intercal},({e}_U,{e}_V)^{\intercal} \rrangle =
   \langle Ae_{V},e_U \rangle- \langle Ae_U,e_V \rangle
 -\langle A{R}_U,e_{U} \rangle- \langle {R}_V,{e}_V\rangle,
 \end{aligned}
\end{equation*}
i.e.,
\begin{equation*}
  \displaystyle   \frac{1}{2} \frac{d}{dt} |\Vert (e_U,e_V) \Vert|^2 =  \llangle -({R}_U,R_{V})^{\intercal},({e}_U,{e}_V)^{\intercal} \rrangle.
\end{equation*}
By the Cauchy--Schwarz inequality, we obtain
\begin{equation}\label{inequality1}
 \displaystyle  \frac{1}{2} \frac{d}{dt} |\Vert (e_U,e_V) \Vert|^2\leq  |\Vert ({R}_U,R_{V}) \Vert| |\Vert ({e}_U,{e}_V) \Vert|.
\end{equation}
Denoting $|\Vert (e_U,e_V)(\tau) \Vert|=\max_{t\in [0,t^{N}]} |\Vert (e_U,e_V)(t) \Vert|$,  we integrate inequality \eqref{inequality1} from $0$ to $\tau$ and obtain
\begin{equation}\label{inequality2}
 \displaystyle \frac{1}{2} |\Vert (e_U,e_V)(\tau)\Vert|^2 \leq    \frac{1}{2}  |\Vert (e_U,e_V)(0)\Vert|^2  + |\Vert (e_U,e_V)(\tau) \Vert | \int_{0}^{\tau} |\Vert ({R}_U,R_{V}) (s)\Vert|ds.
\end{equation}
It follows from the Young inequality that
\begin{equation*}
 \displaystyle |\Vert (e_U,e_V)(\tau)\Vert|^2 \leq
\frac{1}{1-2\theta}|\Vert (e_U,e_V)(0)\Vert|^2 +
 \frac{1}{2\theta-4\theta^2} \bigg( \int_{0}^{\tau} |\Vert ({R}_U,R_{V}) (s)\Vert|ds\bigg)^2,
\end{equation*}
where the constant $0<\theta<1/2$. The proof is complete.
\end{proof}

\begin{remark}
Theorem \ref{suboptimal error bound} gives the upper  bound for the IF midpoint method for linear hyperbolic equations.
Although the error $e=z-Z$ is of second order,  residual-based a posteriori error estimate is of suboptimal (first) order.
 To derive optimal order error bounds, we introduce an appropriate reconstruction $\hat{Z}=(\hat{U},\hat{V})^{\intercal}:[0,T]\rightarrow X$
 of the approximation  $Z$ and
  analyze the optimality of the a posteriori error estimator in the following section.
\end{remark}

\section{Optimal a posteriori error estimate} In this section we will derive optimal order a posteriori error estimates for the IF midpoint method for linear hyperbolic equations.
\subsection{Reconstruction}\label{sec:reconstruction}
Let us define the linear interpolation $\psi(t):J_{n}\rightarrow X$ of  $e^{-\frac{1}{2}k_nM}\bar{f}$ at the nodes $t^{n-\frac{1}{2}}$ and $t^{n-1}$,
\begin{equation}
  \psi(t):=e^{-\frac{1}{2}k_nM}\Big(\bar{f}^{n-\frac{1}{2}}+\frac{2}{k_n}(t-t^{n-\frac{1}{2}})\big(\bar{f}^{n-\frac{1}{2}}-\bar{f}^{n-1}\big)\Big), \quad t\in J_{n},
\end{equation}
with $\bar{f}^{n-\frac{1}{2}}:=\bar{f}(t^{n-\frac{1}{2}})$, and introduce the piecewise quadratic polynomial  $\Psi(t):=\int_{t^{n-1}}^{t} \psi(s)ds$,
\begin{equation*}
 \Psi(t)=e^{-\frac{1}{2}k_nM}\Big((t-t^{n-1})\bar{f}^{n-\frac{1}{2}}+\frac{1}{k_n}(t-t^{n-1})(t-t^{n})\big(\bar{f}^{n-\frac{1}{2}}-\bar{f}^{n-1}\big)\Big).
\end{equation*}
Importantly, $\Psi$ has the property that
\begin{equation*}
\Psi(t^{n-1})=0,\quad \Psi(t^{n})=k_ne^{-\frac{1}{2}k_nM}\bar{f}^{n-\frac{1}{2}}=\int_{t^{n-1}}^{t^{n}} \psi(s)ds.
\end{equation*}
For any $t\in J_{n}$, we are  now in a position to introduce the  reconstruction $\hat{Z}$ of $Z$ by
\begin{equation}\label{reconstruction}
\begin{aligned}
\hat{Z}(t):&=Z^{n-1}-\int_{t^{n-1}}^{t}e^{-\frac{1}{2}k_nM}MZ(s)ds+\Psi(t)+\big(I-e^{-\frac{1}{2}k_nM}\big)\int_{t^{n-1}}^{t} Z^{\prime}(s)ds\cr\noalign{\vskip2truemm}
&=Z^{n-1}-\int_{t^{n-1}}^{t}e^{-\frac{1}{2}k_nM}MZ(s)ds+\Psi(t)+(t-t^{n-1})\big(I-e^{-\frac{1}{2}k_nM}\big)\bar{\partial}
Z^{n},
\end{aligned}
\end{equation}
 Evaluating   the integral of \eqref{reconstruction}  by  the following approximation
\begin{equation*}
\begin{aligned}
  \hat{Z}(t)&=Z^{n-1}-(t-t^{n-1})e^{-\frac{1}{2}k_nM}M(\tau_1Z^{n-1}+\tau_2Z^{n})\cr\noalign{\vskip3truemm}
  &\quad +\Psi(t)+(t-t^{n-1})\big(I-e^{-\frac{1}{2}k_nM}\big)\bar{\partial}Z^{n}, \quad \forall t \in J_{n},
\end{aligned}
\end{equation*}
where the constants $\tau_1=-k_n^{-1}(e^{-\frac{1}{2}k_nM}M)^{-1}\big[I-e^{-\frac{1}{2}k_nM}-k_n\varphi_1(-k_nM)M\big]$, and
$\tau_2=k_n^{-1}(e^{-\frac{1}{2}k_nM}M)^{-1}\big[I-e^{-\frac{1}{2}k_nM}\big]$, thus the reconstruction $\hat{Z}$ takes the form
{\small{
\begin{equation}\label{reconstruction2}
\begin{aligned}
  \hat{Z}(t)&=Z^{n-1}+\frac{t-t^{n-1}}{k_n}\bigg[ \Big(I-e^{-\frac{1}{2}k_nM}-k_n\varphi_1(-k_nM)M\Big)Z^{n-1}-\big(I-e^{-\frac{1}{2}k_nM}\big)Z^{n}\bigg]  \cr\noalign{\vskip2truemm}
&\quad +\Psi(t)+\frac{t-t^{n-1}}{k_n}\big(I-e^{-\frac{1}{2}k_nM}\big)(Z^{n}-Z^{n-1})  \cr\noalign{\vskip2truemm}
  &=Z^{n-1}-(t-t^{n-1})\varphi_1(-k_nM)MZ^{n-1}+ \Psi(t).
  \end{aligned}
\end{equation}}}
It is simple to verify that $\hat{Z}(t^{n-1})=Z^{n-1}$ and
\begin{equation*}
  \hat{Z}(t^{n})=Z^{n-1}-k_n\varphi_1(-k_nM)MZ^{n-1}+k_ne^{-\frac{1}{2}k_nM}\bar{f}^{n-\frac{1}{2}}=Z^{n}.
\end{equation*}
Therefore, the reconstruction $\hat{U}$ is coincided with the approximation $U$ at the nodes $\{t^{n}\}_{n=0}^{N}$; in particular,  the reconstruction $\hat{U}:[0,T]\rightarrow X$ is continuous.

 The residual $\hat{R}(t)=(\hat{R}_{\hat{U}}(t),\hat{R}_{\hat{V}}(t))^{\intercal}$ of the reconstruction $\hat{Z}$ is defined by
 \begin{equation*}
   \hat{R}(t):=\hat{Z}^{\prime}(t)+M\hat{Z}(t)-\bar{f}(t), \quad  t\in J_n.
 \end{equation*}
By \eqref{reconstruction},  the reconstruction $\hat{Z}$ satisfies that
\begin{equation*}
  \hat{Z}^{\prime}(t)=-e^{-\frac{1}{2}k_nM}MZ(t)+\psi(t)+\big(I-e^{-\frac{1}{2}k_nM}\big)Z^{\prime}(t),\quad \forall t\in J_{n},
\end{equation*}
and thus we have
\begin{equation*}
  \hat{R}(t)=-e^{-\frac{1}{2}k_nM}MZ(t)+\psi(t)+(I-e^{-\frac{1}{2}k_nM})Z^{\prime}(t)+M\hat{Z}(t)-\bar{f}(t),\quad t\in J_{n}.
\end{equation*}
The residual $\hat{R}$ can be rewritten as
\begin{equation}\label{second order residual}
\begin{aligned}
  \hat{R}(t)&=M(\hat{Z}(t)-Z(t))+\psi(t)+\big(I-e^{-\frac{1}{2}k_nM}\big)\big(Z^{\prime}(t)+MZ(t)\big)-\bar{f}(t)\cr\noalign{\vskip2truemm}
  &=M\big(\hat{Z}(t)-Z(t)\big)+ \big(I-e^{-\frac{1}{2}k_nM}\big)\big(Z^{\prime}(t)+MZ(t)-\bar{f}(t)\big)
  \cr\noalign{\vskip2truemm}
  &\quad+e^{-\frac{1}{2}k_nM}\Big(\bar{f}^{n-\frac{1}{2}}+\frac{2}{k_n}(t-t^{n-\frac{1}{2}})
  (\bar{f}^{n-\frac{1}{2}}-\bar{f}^{n-1})-\bar{f}(t)\Big), \quad t\in J_{n}.
  \end{aligned}
\end{equation}

\begin{remark}[Optimality of the residual $\hat{R}$] In light of \eqref{linearinterpolation} and \eqref{reconstruction2},   the a posteriori quantity  $\hat{Z}(t)-Z(t)$ is
\begin{equation*}
\begin{aligned}
\hat{Z}(t)-Z(t)&= Z^{n-1}-(t-t^{n-1})\varphi_1(-k_nM)MZ^{n-1}+\Psi(t)-Z(t)\cr\noalign{\vskip2truemm}
&=\frac{1}{k_n} (t-t^{n-1})(t-t^{n}) e^{-\frac{1}{2}k_nM}\big( \bar{f}^{n-\frac{1}{2}}-\bar{f}^{n-1}\big), \quad t \in J_{n},
\end{aligned}
\end{equation*}
i.e.,
\begin{equation*}
\begin{aligned}
\left(
                   \begin{array}{c}
                    \hat{U}(t)-U(t)\\
                     \hat{V}(t)-V(t)
                   \end{array}
                 \right)
= \frac{1}{k_n}(t-t^{n-1})(t-t^{n})e^{-\frac{1}{2}k_nM} \left(
    \begin{array}{c}
  0\\
 f(t^{n-\frac{1}{2}})-f(t^{n-1})
    \end{array}
  \right), \quad t\in J_{n}.
\end{aligned}\end{equation*}
For the uniformly bounded exponential operator $e^{-tM}$, the a posteriori quantities $\hat{Z}(t)-Z(t)$  and $e^{-\frac{1}{2}k_nM}\big(\bar{f}^{n-\frac{1}{2}}+\frac{2}{k_n}(t-t^{n-\frac{1}{2}})
  (\bar{f}^{n-\frac{1}{2}}-\bar{f}^{n-1})-\bar{f}(t)\big)$ are of second order once $\bar{f}\in C^{2}(0,T;X)$. Due to \eqref{linear residual} and \eqref{exponential property1}, the a posteriori quantity $(I-e^{-\frac{1}{2}k_nM})\big(Z^{\prime}(t)+MZ(t)-\bar{f}(t)\big)$ is also second order. Consequently, the residual $\hat{R}$ is of optimal (second) order.
\end{remark}

\subsection{Error analysis}\label{sec:optimal error estimates}
Setting the errors $e:=z-Z=(e_U,e_V)^{\intercal}$ and $\hat{e}:=z-\hat{Z}=(\hat{e}_{U},\hat{e}_{V})^{\intercal}$, and according to the fact that
\begin{equation*}\left\{
\begin{array}{l}
\displaystyle
z^{\prime}(t)+Mz(t)=\bar{f}(t),\cr\noalign{\vskip2truemm}
\displaystyle
\hat{Z}^{\prime}(t)+e^{-\frac{1}{2}k_nM}MZ(t)=\psi(t)+\big(I-e^{-\frac{1}{2}k_nM}\big)Z^{\prime}(t),
\end{array}
\right.
\end{equation*}
it follows that
\begin{equation}\label{errorequa}
\begin{aligned}
\hat{e}^{\prime}(t)+Me(t)&=\bar{f}(t)-\psi(t)-\big(I-e^{-\frac{1}{2}k_nM}\big)\big(Z^{\prime}(t)+MZ(t)\big)\cr\noalign{\vskip2truemm}
&=e^{-\frac{1}{2}k_nM}\bar{f}(t)-\psi(t)+\big(e^{-\frac{1}{2}k_nM}-I\big)R(t),
\end{aligned}
\end{equation}
where $R(t)$ is the residual of $Z$, and  the a posteriori quantity
$e^{-\frac{1}{2}k_nM}\bar{f}(t)-\psi(t)$ is
\begin{equation}\label{error of linear interpolation}
e^{-\frac{1}{2}k_nM}\bar{f}(t)-\psi(t)=e^{-\frac{1}{2}k_nM} \Big(\bar{f}(t)-\bar{f}^{n-\frac{1}{2}}-\frac{2}{k_n}(t-t^{n-\frac{1}{2}})(\bar{f}^{n-\frac{1}{2}}-\bar{f}^{n-1})\Big), \quad t\in J_n.
\end{equation}

For the sake of convenience, we denote  the a posteriori quantities $(R_{\bar{f}_1},R_{\bar{f}_2})^{\intercal}:=e^{-\frac{1}{2}k_nM}\bar{f}(t)-\psi(t)$ and $(\bar{R}_U,\bar{R}_V)^{\intercal}:=(e^{-\frac{1}{2}k_nM}-I)R(t)$; thus we have
\begin{equation*}
\begin{aligned}
\left(    \begin{array}{c}
                    \hat{e}_U\\
                     \hat{e}_V
                   \end{array}
                 \right)^{\prime}
                 +
     \left(    \begin{array}{c}
                   - {e}_V\\
                     Ae_U
                   \end{array}
                 \right)
=
\left(    \begin{array}{c}
                    R_{\bar{f}_1}\\
                    R_{\bar{f}_2}
                   \end{array}
                 \right)
                 +
                \left(
    \begin{array}{c}
  \bar{R}_{U}\\
  \bar{R}_{V}
    \end{array}
  \right),
\end{aligned}\end{equation*}
and the error equation can be rewritten as
 \begin{equation}\label{use4}
\begin{aligned}
\left(    \begin{array}{c}
                    \hat{e}_U\\
                     \hat{e}_V
                   \end{array}
                 \right)^{\prime}
                =
     \left(    \begin{array}{c}
                   \hat{e}_V\\
                    -A\hat{e}_U
                   \end{array}
                 \right)
+   \left(    \begin{array}{c}
                   \hat{V}-V\\
                    A(U-\hat{U})
                   \end{array}
                 \right) +
\left(    \begin{array}{c}
                     R_{\bar{f}_1}\\
                    R_{\bar{f}_2}
                   \end{array}
                 \right)
              +
                \left(
    \begin{array}{c}
  \bar{R}_{U}\\
  \bar{R}_{V}
    \end{array}
  \right).
\end{aligned}\end{equation}
A further regularity assumption is imposed on $\hat{Z}$ defined in \eqref{reconstruction}:
\begin{equation*}
  \hat{Z}(t)\in D(M), \quad \forall t \in [0,T].
\end{equation*}

\begin{theorem}\label{optimal error bound}
Suppose that  $u$ is the exact solution of  \eqref{linearhyperbolicequation},  $z=(u,v)^{\intercal}$ satisfies
 \eqref{linearproblem}, and the IF midpoint approximation $Z=(U,V)^{\intercal}$
and its piecewise quadratic time reconstruction  $\hat{Z}=(\hat{U},\hat{V})^{\intercal}$ are respectively defined by \eqref{linearinterpolation}
and  \eqref{reconstruction}.  Denoting the error by $\hat{e}:=z-\hat{Z}=(\hat{e}_{U},\hat{e}_{V})^{\intercal}$,
the following a posteriori error estimate holds for any $0<\theta<1/2$:
\begin{equation*}
\begin{aligned}
 \displaystyle \max_{t\in[0,t^{N}]}|\Vert& (\hat{e}_U, \hat{e}_V)(t)\Vert|^2\leq \frac{1}{1-2\theta}  |\Vert (\hat{e}_U,\hat{e}_V)(0)\Vert|^2 \cr\noalign{\vskip0.2truemm}
 \displaystyle   + &\  \frac{1}{2\theta-4\theta^2} \bigg(\int_{0}^{t^{N}} |\Vert (\hat{V}-V+R_{\bar{f}_1}+\bar{R}_{U}, A(U-\hat{U})+R_{\bar{f}_2}+\bar{R}_{V}) (s)\Vert|ds\bigg)^2,
   \end{aligned}
\end{equation*}
where  $(\bar{R}_U,\bar{R}_V)^{\intercal}=(e^{-\frac{1}{2}k_nM}-I)R$, and
$(R_{\bar{f}_1},R_{\bar{f}_2})^{\intercal}$ is defined by \eqref{error of linear interpolation}.
\end{theorem}

\begin{proof}
We simplify the error equation \eqref{use4} into a compact form
 \begin{equation}\label{use5}
\begin{aligned}
\left(    \begin{array}{c}
                    \hat{e}_U\\
                     \hat{e}_V
                   \end{array}
                 \right)^{\prime}
                =
     \left(    \begin{array}{c}
                   \hat{e}_V\\
                    -A\hat{e}_U
                   \end{array}
                 \right)
+   \left(    \begin{array}{c}
                   \hat{V}-V+R_{\bar{f}_1}+\bar{R}_{U}\\
                    A(U-\hat{U})+R_{\bar{f}_2}+\bar{R}_{V}
                   \end{array}
                 \right).
\end{aligned}\end{equation}
Taking the inner product of equation \eqref{use5} with  $(\hat{e}_U,\hat{e}_V)^{\intercal}$ leads to
\begin{equation*}
\begin{aligned}
   \displaystyle \llangle (\hat{e}_U^{\prime},\hat{e}_V^{\prime})^{\intercal},(\hat{e}_U,\hat{e}_V&)^{\intercal} \rrangle =
 \llangle (\hat{e}_V,-A\hat{e}_U)^{\intercal},(\hat{e}_U,\hat{e}_V)^{\intercal} \rrangle \cr\noalign{\vskip3truemm}
 \displaystyle &\ +\llangle (\hat{V}-V+R_{\bar{f}_1}+\bar{R}_{U}, A(U-\hat{U})+R_{\bar{f}_2}+\bar{R}_{V})^{\intercal},(\hat{e}_U,\hat{e}_V)^{\intercal}\rrangle.
 \end{aligned}
\end{equation*}
By using the self-adjointness of $A$, we get
\begin{equation*}
\begin{aligned}
 \displaystyle\llangle (\hat{e}_U^{\prime},\hat{e}_V^{\prime})^{\intercal},(\hat{e}_U,\hat{e}_V & )^{\intercal}\rrangle
 =\langle A\hat{e}^{\prime}_U,\hat{e}_U  \rangle +\langle  \hat{e}^{\prime}_V,\hat{e}_V \rangle\cr\noalign{\vskip1truemm}
  \displaystyle&= \langle A\hat{e}_V, \hat{e}_U \rangle - \langle A\hat{e}_U,\hat{e}_V \rangle+\langle  A(\hat{V}-V+R_{\bar{f}_1}+\bar{R}_{U}), \hat{e}_U \rangle \cr\noalign{\vskip1truemm}
 \displaystyle &\quad +\langle A(U-\hat{U})+R_{\bar{f}_2}+\bar{R}_{V}, \hat{e}_V \rangle \cr\noalign{\vskip1truemm}
 \displaystyle &= \langle  A(\hat{V}-V+R_{\bar{f}_1}+\bar{R}_{U}), \hat{e}_U \rangle  +\langle A(U-\hat{U})+R_{\bar{f}_2}+\bar{R}_{V}, \hat{e}_V \rangle,
 \end{aligned}
\end{equation*}
i.e.,
\begin{equation}\label{use6}
\frac{1}{2} \frac{d}{dt} |\Vert (\hat{e}_U,\hat{e}_V)\Vert|^2=\llangle (\hat{V}-V+R_{\bar{f}_1}+\bar{R}_{U}, A(U-\hat{U})+R_{\bar{f}_2}+\bar{R}_{V})^{\intercal},(\hat{e}_U,\hat{e}_V)^{\intercal}\rrangle.
\end{equation}
Applying the  Cauchy--Schwarz inequality to \eqref{use6} yields
\begin{equation*}
\frac{1}{2} \frac{d}{dt} |\Vert (\hat{e}_U,\hat{e}_V)\Vert|^2 \leq  |\Vert (\hat{V}-V+R_{\bar{f}_1}+\bar{R}_{U}, A(U-\hat{U})+R_{\bar{f}_2}+\bar{R}_{V}) \Vert|
|\Vert (\hat{e}_U,\hat{e}_V) \Vert|.
\end{equation*}

Following a similar argument to that for deriving \eqref{inequality2}, we deduce that
\begin{equation*}
\begin{aligned}
 \displaystyle\frac{1}{2}  |\Vert (\hat{e}_U,\hat{e}_V)(\tau)\Vert|^2  &\leq  \frac{1}{2}  |\Vert (\hat{e}_U,\hat{e}_V)(0)\Vert|^2  + |\Vert (\hat{e}_U,\hat{e}_V)(\tau) \Vert | \cr\noalign{\vskip0.2truemm}
 \displaystyle&\quad  \times \int_{0}^{t^{N}} |\Vert (\hat{V}-V+R_{\bar{f}_1}+\bar{R}_{U}, A(U-\hat{U})+R_{\bar{f}_2}+\bar{R}_{V}) (s)\Vert|ds.
\end{aligned}
\end{equation*}
By the Young inequality, we have
\begin{equation*}
\begin{aligned}
  |\Vert (\hat{e}_U, & \hat{e}_V)(\tau)\Vert|^2 \leq   \frac{1}{1-2\theta}  |\Vert (\hat{e}_U,\hat{e}_V)(0)\Vert|^2 \cr\noalign{\vskip0.2truemm}
 &+ \frac{1}{2\theta-4\theta^2}\bigg(\int_{0}^{t^{N}} |\Vert (\hat{V}-V+R_{\bar{f}_1}+\bar{R}_{U}, A(U-\hat{U})+R_{\bar{f}_2}+\bar{R}_{V}) (s)\Vert|ds\bigg)^2,
\end{aligned}
\end{equation*}
with the constant $0<\theta<1/2$. This completes the proof.
\end{proof}

  \begin{remark}Theorem  \ref{optimal error bound} presents optimal order a posteriori error estimate 
  for the IF midpoint method for \eqref{linearhyperbolicequation} in the $L^\infty$-in-time and energy-in-space norms.  The upper error bound depends only on the discretization parameters and the data of the problems, and is independent of
the unknown exact solutions; therefore, it is computable.
\end{remark}

\begin{remark}
Error bound established in this paper, involves  exponential functions; however, the practical computation of the products of matrix exponentials and vectors with larger time stepsizes  may slightly affect the convergence rate of the a posteriori quantity. In addition, the $\varphi$-functions are employed solely to the analyze the residual $R(t)$ and introduce the reconstruction $\hat{Z}$, entailing no computation in numerical experiments.
\end{remark}

\begin{remark}
From \eqref{reconstruction2}, the reconstruction $\hat{Z}=(\hat{U},\hat{V})^{\intercal}$ and the IF midpoint approximation $Z=(U,V)^{\intercal}$ are continuous at the nodes $t^{0},\ldots,t^{N}$,  then for  $0<\theta<1/2$,  the error  $e=(e_U,e_V)^{\intercal}$ satisfies
\begin{equation*}
\begin{aligned}
 \displaystyle \max_{ 0\leq n\leq N }&|\Vert ({e}_U,{e}_V)(t^{n})\Vert|^2 \leq   \frac{1}{1-2\theta}|\Vert (e_U,e_V)(0)\Vert|^2\cr\noalign{\vskip0.2truemm}
 \displaystyle  &\  +\frac{1}{2\theta-4\theta^2} \bigg(\int_{0}^{t^{N}} |\Vert (\hat{V}-V+R_{\bar{f}_1}+\bar{R}_{U}, A(U-\hat{U})+R_{\bar{f}_2}+\bar{R}_{V}) (s)\Vert|ds\bigg)^2.
   \end{aligned}
\end{equation*}
\end{remark}

\begin{remark} It is of some interest to note that the error $|\Vert (\hat{e}_U, \hat{e}_V)(t)\Vert|$  is bounded by the initial error $|\Vert (\hat{e}_U,\hat{e}_V)(0)\Vert|$ and the residual $\bigg(\int_{0}^{t^{N}} |\Vert (\hat{V}-V+R_{\bar{f}_1}+\bar{R}_{U}, A(U-\hat{U})+R_{\bar{f}_2}+\bar{R}_{V}) (s)\Vert|ds\bigg)$. When the initial approximation $Z^{0}=(U^{0},V^{0})^{\intercal}$ is taken as the exact initial data $(u^{0},v^{0})^{\intercal}$, the error will be completely determined by the residual. 
\end{remark}

\section{An adaptive algorithm}\label{sec:an adaptive algorithm}
We briefly describe the adaptive IF midpoint method for linear hyperbolic equations on the interval $J_{n}$. 
Given the tolerance error $Tol$, the initial time step-size $k_0$, 
and the parameters $\delta_{1}$, $\delta_2$, $\delta_3\in(0,1)$. 
Let $k_{\max}$ denote the maximum time step-size and $Count$ denote the iteration counter. The a posteriori error estimator is required to satisfy
 \begin{equation}\label{tolerance}
 \begin{aligned}
Tol&\geq \frac{1}{1-2\theta}|\Vert (e_U,e_V)(0)\Vert|^2\cr\noalign{\vskip0.2truemm}
&+\frac{1}{2\theta-4\theta^2} \Big(\sum_{n=1}^{N}\int_{t^{n-1}}^{t^{n}} |\Vert (\hat{V}-V+R_{\bar{f}_1}+\bar{R}_{U}, A(U-\hat{U})+R_{\bar{f}_2}+\bar{R}_{V}) (s)\Vert|ds\Big)^2.
 \end{aligned}
 \end{equation}
Let $\mathcal{E}_{\theta}^{n}$,
\begin{equation*}
\begin{aligned}
\mathcal{E}_{\theta}^{n}:&=\frac{1}{T\sqrt{1-2\theta}}|\Vert (e_U,e_V)(0)\Vert|\cr\noalign{\vskip0.2truemm}
&+\frac{1}{k_n\sqrt{(2\theta-4\theta^2)}}\int_{t^{n-1}}^{t^{n}} |\Vert (\hat{V}-V+R_{\bar{f}_1}+\bar{R}_{U}, A(U-\hat{U})+R_{\bar{f}_2}+\bar{R}_{V}) (s)\Vert|ds,
\end{aligned}
\end{equation*} be the local error estimator.  If $\mathcal{E}_{\theta}^{n}\leq \eta$, where $\eta:=\frac{1}{T}\sqrt{Tol}$ and $\theta \in (0,\frac{1}{2})$, then inequality \eqref{tolerance} holds.
Our time-stepping strategy is based on the local error estimator $\mathcal{E}_{\theta}^{n}$,   we compute $\mathcal{E}_{\theta}^{n}$ and check $\mathcal{E}_{\theta}^{n}\leq \eta$ for each time step (cf., e.g., \cite{Chen2004,Huang2013}). Note that the initial error $|\Vert (e_U,e_V)(0)\Vert|$  only needs to be calculated once during the entire adaptive process for computing $\mathcal{E}_{\theta}^{n}$. When $\mathcal{E}_{\theta}^{n}\leq \delta_1 \eta$, the current step-size $k_n$ is highly suitable, and we try to increase the time step-size $k_{n+1}=\min(k_{\max},k_n/{\delta_2})$ for the subsequent step; when $\delta_1 \eta <\mathcal{E}_{\theta}^{n}\leq \eta$, the current step-size $k_n$ is acceptable, and we retain $k_{n+1}=k_{n}$; otherwise, the current step-size $k_n$ is too large  and we reduce it to $k_n=\delta_3k_n$, and the discrete solution $Z^{n}$ along with the error estimator $\mathcal{E}_{\theta}^{n}$ are recomputed. Setting $\delta_1={1}/{4}$,  $\delta_2={2}/{3}$, and $\delta_3={1}/{2}$,  the pseudocode of the adaptive IF midpoint method is as follows.
\begin{algorithm}
\caption{Time step-size control}
\label{alg:time step-size}
\begin{algorithmic}
\STATE{ Choose parameters: $Tol$,  $\theta$, $k_0$, $k_{\max}$}
\STATE{Initialization: $Z^{0}:=(U^0,V^{0})^{\intercal}$,  $Count=0$, $t^{0}=0$}
\STATE{Set $k_n:=k_{n-1}$}
\STATE{Compute $Z^{n}$ and  $\mathcal{E}_{\theta}^{n}$}
\IF{$\mathcal{E}_{\theta}^{n}\leq \delta_1\eta$}
\STATE{$t^{n}:=t^{n-1}+k_n$}
\STATE{$Count=Count+1$}
\STATE{$k_{n+1}=min(k_{\max},k_n/{\delta_2}$})
\ELSIF{$\delta_1 \eta <\mathcal{E}_{\theta}^{n}\leq \eta$}
\STATE{$t^{n}:=t^{n-1}+k_{n}$}
\STATE{$Count=Count+1$}
\STATE{$k_{n+1}=k_{n}$}
\ELSE
\STATE{$k_n=\delta_3 k_{n}$}
\ENDIF
\end{algorithmic}
\end{algorithm}

\begin{remark}  From Algorithm \ref{alg:time step-size}, we present an adaptive step-size IF midpoint method \eqref{IFmethod} for solving second order evolution equations, using the parameters  $Tol$, $\theta$,  $k_0$, and $k_{\max}$. Here the tolerance $Tol$ regulates the local error estimator $\mathcal{E}_{\theta}^{n}$, thereby determining the accuracy of  numerical solutions $Z^{n}$;
the parameter $k_{\max}$ constrains the maximum time step-size to prevent numerical instability induced by excessive single-step growth.
\end{remark}

\section{Numerical experiments}\label{sec:numerical experiments}
In this section, we implement several numerical examples to confirm the correct convergence rate of the a posteriori error estimator and the effectiveness of the adaptive algorithm. The matrix exponential $e^{-t\mathbf{M}}$ is computed by the \emph{expm} function in MATLAB R2024a.

\begin{example}\label{example1} We first consider the following hyperbolic equation
    \begin{equation}\label{problem1}
\left\{
\begin{aligned}
  \displaystyle   \frac{\partial^2u}{\partial t^2}&=\frac{\partial^2u}{\partial x^2}+f(x,t), & x \in [0,1],  \quad t\in[0,1], \cr\noalign{\vskip2truemm}
  \displaystyle   u(0,t)&=u(1,t)=0, & t\in[0,1],  \cr\noalign{\vskip2truemm}
  \displaystyle   u(x,0)&=u^{0}(x), \quad \frac{\partial u}{\partial t}(x,0)=v^{0}(x), &x\in [0,1],
\end{aligned}
   \right.
    \end{equation}
where the functions $f$, $u^{0}(x)$ and $v^{0}(x)$ are chosen such that the exact solution of  \eqref{problem1} is $u(x,t)=\sin(\pi x)e^{t}$.
\end{example}

The spatial derivative  is discretized  by  the second order finite difference, and we obtain the following first order system
\[\frac{d}{dt}\left(\begin{array}[c]{c}
\bar{U}\\
\bar{V}\\
\end{array}\right)+\left(\begin{array}[c]{cc}
{\bf 0}_{m\times m}&-\mathbf{I}_{m\times m}\\
\mathbf{A}_{m\times m}&{\bf 0}_{m\times m}\\
\end{array}\right)\left(\begin{array}[c]{c}
\bar{U}\\
\bar{V}\\
\end{array}\right)=\left(\begin{array}[c]{c}
0^{\intercal}\\
F\\
\end{array}\right),
\]
where  $\mathbf{I}_{m\times m}$ is the $m\times m$ identity matrix,
\begin{equation*}
\bar{U}(t)=(u_1(t),\cdots,u_m(t))^{\intercal}, \  \bar{V}(t)=(v_1(t),\cdots,v_m(t))^{\intercal}, \   F(t)=(f_1(t),\cdots,f_m(t))^{\intercal},
\end{equation*}
with  $u_i(t) \approx u(x_i,t)$, $v_i(t)\approx u^{\prime}(x_i,t)$, $i=1,\ldots,m$,  $x_i=i\Delta x$, $\Delta x=\frac{1}{m+1}$,  and
\[\mathbf{A}_{m\times m}=\dfrac{1}{\Delta x^2}
\left(
\begin{array}
[c]{ccccc}
2 & -1 &  &  & \\
-1 & 2&  -1 &  &  \\
& \ddots & \ddots & \ddots &   \\
  &  & -1 & 2 & -1\\
  &  &   &  -1&2  \\
\end{array}
\right)_{m \times m}.
\]
Denoting $Err_{\infty}$ by the error in  the discrete maximum norm,
\begin{equation*}
Err_{\infty}:= \max_{ 1\leq n\leq N }|\Vert ({e}_U,{e}_V)(t^{n})\Vert|= \max_{ 1\leq n\leq N }\Big(\langle A{e}_U(t^{n}), {e}_U (t^{n})\rangle+ \langle {e}_V(t^{n}),{e}_V(t^{n})\rangle \Big)^{\frac{1}{2}}.
\end{equation*}
The a posteriori quantity  $\mathcal{E}$  is  defined by
\begin{equation*}
\mathcal{E}:= \sum_{n=1}^{N} \int_{t^{n-1}}^{t^{n}} |\Vert (\hat{V}-V+R_{\bar{f}_1}+\bar{R}_{U}, A(U-\hat{U})+R_{\bar{f}_2}+\bar{R}_{V}) (s)\Vert|ds.
\end{equation*}
For simplicity, we take the Gauss–Legendre quadrature formula with three nodes to approximate the integral between $t^{n-1}$ and  $t^n$.

For the temporal accuracy test, we take  the spatial mesh size   $\Delta x=1/2500$ and a uniform temporal step-size $k=1/N$, $N=16,32,\ldots,512$. According to the a posteriori error analysis established in  Theorem \ref{optimal error bound}, the upper error bound of the method \eqref{IFmethod} is $5\mathcal{E}$ when  $\theta=\frac{5-\sqrt{21}}{20}\in (0,\frac{1}{2})$, and we introduce the effectivity index of the upper estimator $ei_U:=\frac{5\mathcal{E}}{Err_{\infty}}$. Table \ref{table1-1} presents  second order convergence rates of the error  $Err_{\infty}$  and  the a posteriori quantity $\mathcal{E}$. From Table \ref{table1-1}, we observe that
the effectivity index $ei_U$ seems to be asymptotically constant (around $1.0$).

\begin{table}
\centering
\caption{Order of the exact error $Err_{\infty}$ and  the a posteriori quantity $\mathcal{E}$,   and  the effectivity index for the method \eqref{IFmethod} for Example \ref{example1}}\label{table1-1}
\renewcommand\arraystretch{1.3}
  \begin{tabular}{ccccccc}
 \hline
    N& $Err_{\infty}$&  order  & $\mathcal{E}$ & order & $5\mathcal{E}$ & $ei_U$ \\
   \hline
16&1.5352e-02  &       &4.5184e-03 &          &2.2592e-02 & 1.4716\\
32&3.8357e-03  &2.0009	&9.4953e-04 &2.2505   &4.7476e-03 & 1.2378\\
64&9.5937e-04  &1.9993   &2.1692e-04 &2.1300  &1.0846e-03 & 1.1305\\
128& 2.4047e-04&1.9963  &5.1730e-05 &2.0681   &2.5865e-04 &  1.0756\\
256&6.0752e-05&1.9848   &1.2648e-05 &2.0321   &6.3240e-05 &1.0410 \\
512&1.5832e-05&1.9401   &3.1711e-06 & 1.9959  &1.5856e-05 & 1.0015 \\
 \hline
       \end{tabular}
\end{table}

\begin{example}\label{example2}
We then  consider  the another linear  hyperbolic equation \eqref{problem1} with $u(x,t)$$=\cos(\frac{\pi}{2}x)e^{-10t}$ in the interval  $[-1,1]\times [0,1]$.
\end{example}

Similar to Example \ref{example1}, we employ the center difference scheme with spatial mesh size  $\Delta x=1/800$ to discretize  the spatial derivative.
For a uniform temporal step-size $k=1/N$,  Table \ref{table2-1} presents second-order convergence rates of the error $Err_{\infty}$ and the a posteriori quantity $\mathcal{E}$.  We set the parameter $\theta={1}/{4}$ and the effectivity index of the upper estimator is around $2.7$.  Numerical results are coincided with the theoretical analysis established in Theorem \ref{optimal error bound}.

\begin{table}
\centering
\caption{Order of the exact error $Err_{\infty}$ and  the a posteriori quantity $\mathcal{E}$,   and  the effectivity index for the method \eqref{IFmethod} for Example \ref{example2}}\label{table2-1}
\renewcommand\arraystretch{1.3}
  \begin{tabular}{ccccccc}
 \hline
    N& $Err_{\infty}$&  order  & $\mathcal{E}$ & order  & $2\mathcal{E}$ & $ei_U$ \\
 \hline
16&1.6696e-01 &       &2.3318e-01 &         & 4.6636e-01 &2.7932\\
32&4.2088e-02 &1.9880	&5.7632e-02 &2.0165  &  1.1526e-01 & 2.7386\\
64&1.0544e-02 &1.9970   &1.4302e-02 & 2.0107  &   2.8604e-02 & 2.7128\\
128&2.6373e-03&1.9993  &3.5607e-03 & 2.0060  &  7.1214e-03 &   2.7003\\
256&6.5936e-04&1.9999   &8.8825e-04 &2.0031   &  1.7765e-03 &  2.6943\\
512&1.6479e-04&2.0004   &2.2181e-04 & 2.0016  &   4.4362e-04  &    2.6920  \\
 \hline
       \end{tabular}
\end{table}

\begin{example}\label{example3} In this example, we consider the linear hyperbolic equation \cite{Chen2004,Huang2013}
    \begin{equation}\label{problem3}
\left\{
\begin{aligned}
  \displaystyle  \frac{\partial^2u}{\partial t^2}&=2\frac{\partial^2u}{\partial x^2}+f(x,t),\quad & x \in [0,1],  \quad t\in[0,1], \cr\noalign{\vskip1truemm}
  \displaystyle  u(0,t)&=u(1,t)=0,    &t\in[0,1],  \cr\noalign{\vskip1truemm}
  \displaystyle u(x,0)&=u^{0}(x),  \quad \frac{\partial u}{\partial t}(x,0)=v^{0}(x), & x\in [0,1],
   \end{aligned}
   \right.
    \end{equation}
    where the functions $f$, $u^{0}$, and $v^{0}$  satisfy  that
the exact solution is
\begin{equation}\label{exsolution}
 u(x,t)=0.1\times \sin(\pi x)\times(1-e^{-10000(t-\frac{1}{2})^2}) .
\end{equation}
\end{example}

The time derivative of the solution   \eqref{exsolution} exhibits steep gradients around
the point $t=1/2$, leading to rapid changes in the solution over a
 short time interval. Consequently, a relatively small time step-size
 is required for numerical computations. We use the second order finite
 difference with the spatial mesh size $\Delta x=1/500$ to discrete the
  spatial derivative.  Convergence rates of the error $Err_{\infty}$ and
   the a posteriori quantity $\mathcal{E}$ with a uniform step-size $k=1/N$
   are reported in Table \ref{table3-1}, and  we observe that
   the accuracy and convergence rate  of  the error $Err_{\infty}$
    are sightly affected by the larger time step-size, such as $k=1/128$.
    Setting $\theta={1}/{4}$, the  effectivity index $ei_U$ is around $13.2$.

Then, we apply the adaptive  IF midpoint method with $k_0=1/60$ to Example \ref{example3},
 and present initial  parameters and numerical results in Table  \ref{table3-2},
  which validates the efficiency of the adaptive Algorithm \ref{alg:time step-size}.
  Compared to the adaptive finite element method \cite{Huang2013},
  which computes the error $\|(u-U)^{\prime}\|_{L^{\infty}(0,T)}$, our adaptive algorithm
   generally achieves higher accuracy in the energy norm for numerical solutions.
  For $Tol=0.9$, we plot the time step-size  trajectory and the error at each time step
  in Figure \ref{fig3.1}.  It can be observed that the time step-size drops sharply around
  $t=1/2$, accompanied by a peak of the error within the same region. Once $t>1/2$,
  the error approaches a constant, likely due to  the symmetry of the method \eqref{IFmethod}.
   Moreover, all numerical results satisfy $2\mathcal{E}\leq \sqrt{Tol}$,
   which corresponds to \eqref{tolerance}.
 \begin{table}
\centering
\caption{Order of the exact error $Err_{\infty}$ and  the a posteriori quantity $\mathcal{E}$,   and  the effectivity index for the method \eqref{IFmethod} for Example \ref{example3}}\label{table3-1}
\renewcommand\arraystretch{1.2}
  \begin{tabular}{ccccccc}
    \hline
    N& $Err_{\infty}$&  order & $\mathcal{E}$ & order  & $2\mathcal{E}$ & $ei_U$ \\
    \hline
128& 5.8912e-01 &        &4.1820e+00        &  &8.3641e+00 & 14.1976\\
256& 1.6928e-01 &1.7991  &1.1241e+00& 1.8955  & 2.2481e+00 &   13.2803\\
512&4.3575e-02 &1.9579   &2.8648e-01  &1.9722   & 5.7296e-01 &  13.1488\\
1024&1.0932e-02& 1.9950  &7.2176e-02 & 1.9888  &1.4435e-01 & 13.2046\\
2048&2.7416e-03&1.9954    &1.8079e-02& 1.9972&    3.6158e-02 &13.1884\\
  \hline
       \end{tabular}
\end{table}

 \begin{table}
 \small
\centering
\caption{Initial  parameters and numerical results  for the adaptive step-size IF midpoint method  for Example \ref{example3}}\label{table3-2}
\renewcommand\arraystretch{1.7}
  \begin{tabular}{ccccccccccc}
    \hline
    Tol& $\eta$ & $k_{max}$  & $Err_{\infty}$ &  $2\mathcal{E}$   & $ei_U$& $Count$ \\
    \hline
    9e-1&2.3717e-01&1e-01& 9.4614e-03& 3.0843e-02&3.2598& 159\\
   1e-1&7.9057e-02&8e-02&3.9940e-03&  1.3366e-02&3.3465&241\\
  1e-2& 2.5000e-02&4e-02& 1.0149e-03& 3.9116e-03& 3.8541& 402\\
   1e-3& 7.9057e-03&2e-02&3.2924e-04 &1.2762e-03 & 3.8763  &710\\
    1e-4&2.5000e-03&1e-02&9.2291e-05 & 3.9967e-04&4.3305& 1348\\
    \hline
       \end{tabular}
\end{table}

\begin{figure}[ht]
  \centering
  \begin{subfigure}[c]{0.8\textwidth}
  \includegraphics[width=\textwidth,height=5.4cm]{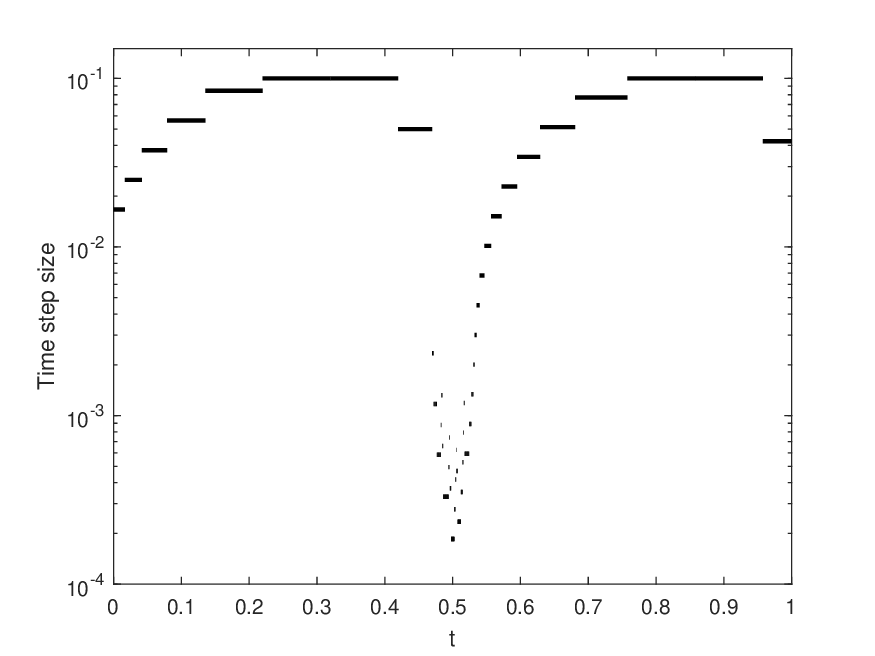}
  \end{subfigure}
      \vspace{0.5cm}
  \begin{subfigure}[c]{0.8\textwidth}
    \includegraphics[width=\textwidth,height=5.4cm]{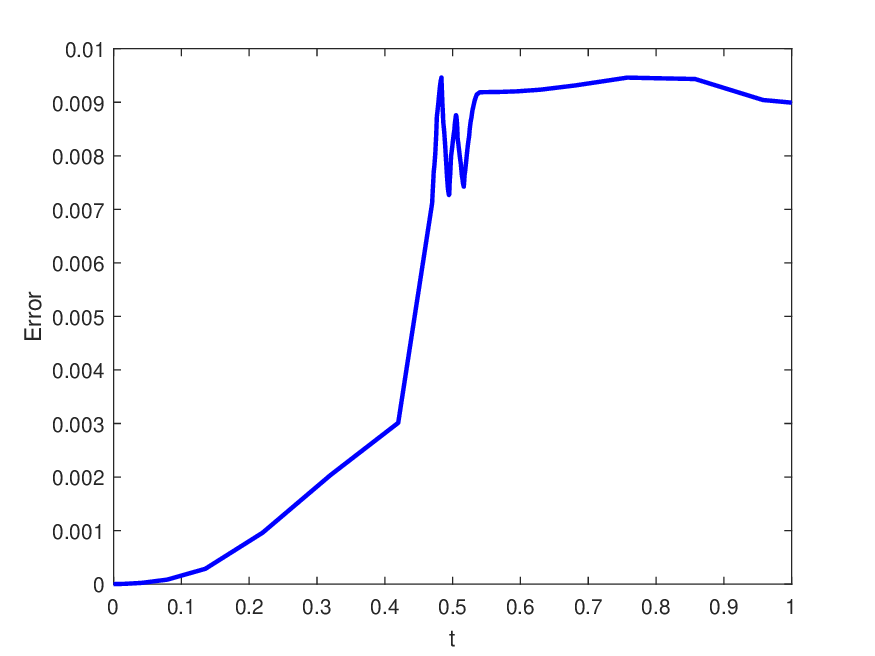}
  \end{subfigure}
    \caption{Example \ref{example3}: Time step-size trajectory (top row) and the error (bottom row) with $Tol=0.9$, $k_0=1/60$, and $k_{\max}=0.1$}\label{fig3.1}
\end{figure}

\begin{example}\label{example4}
In the final experiment, we consider equation \eqref{problem3} with  $u(x,t)=\sin(\pi x)e^{-800(\sin(\pi t /2)-1)^2}\sin(4\pi t)$ defined on the interval  $[0,1]\times [0,10]$.
\end{example}

Following Example \ref{example3}, we discretize the spatial derivative via a central difference scheme with $\Delta x=1/500$ and set the parameter $\theta=1/4$.  The error $Err_{\infty}$ and  the a posteriori quantity $\mathcal{E}$  with a uniform temporal step-size $k=10/N$ are reported in Table \ref{table4-1}. We apply the adaptive Algorithm \ref{alg:time step-size} with $k_0=0.1$ and show the numerical results in Table \ref{table4-2}. The time step-size trajectory and the error at each time step are shown in Figure \ref{fig4.1} when $Tol=0.1$. We observe that the time step-size drops sharply around $t = 1,5,9$, and is accompanied by error peaks within same time intervals; the condition $2\mathcal{E}\leq \sqrt{Tol}$ holds universally in computational results, consistent with \eqref{tolerance}.

All numerical results demonstrate that Algorithm \ref{alg:time step-size}  effectively
detects the singularity and performs local mesh refinement and coarsening.

\begin{table}
\centering
\caption{Order of the exact error $Err_{\infty}$ and  the a posteriori quantity $\mathcal{E}$,   and  the effectivity index for the method \eqref{IFmethod} for Example \ref{example4}}\label{table4-1}
\renewcommand\arraystretch{1.2}
  \begin{tabular}{ccccccc}
    \hline
    N& $Err_{\infty}$&  order & $\mathcal{E}$ & order  & $2\mathcal{E}$ & $ei_U$ \\
    \hline
320& 2.4325e-01 & &7.6526e+00&   &1.5305e+01& 62.9207\\
640&5.8559e-02&2.0545&1.8768e+00 &2.0277  &3.7536e+00 &64.0995\\
1280&1.4657e-02& 1.9983 &4.6954e-01 & 1.9989 &9.3909e-01 & 64.0693 \\
2560&3.6857e-03&1.9916 & 1.1748e-01&    1.9988 & 2.3497e-01&63.7515\\
  \hline
       \end{tabular}
\end{table}

 \begin{table}
 \small
\centering
\caption{Initial  parameters and numerical results  for the adaptive step-size IF midpoint method  for Example \ref{example4}}\label{table4-2}
\renewcommand\arraystretch{1.7}
  \begin{tabular}{ccccccccccc}
    \hline
    Tol& $\eta$ & $k_{max}$  & $Err_{\infty}$ &  $2\mathcal{E}$   & $ei_U$& $Count$ \\
    \hline
    9e-1&2.3717e-02&1.8e-01& 1.4115e-02& 5.1378e-01& 36.3988& 390\\
   1e-1&7.9057e-03&1.2e-01&3.3088e-03&  1.4546e-01& 43.9618&763\\
  1e-2& 2.5000e-03&4e-02& 1.4438e-03& 5.8521e-02&  40.5315& 1253\\
   1e-3& 7.9057e-04&1e-02&3.5124e-04 &1.4442e-02 & 41.1173  &2858\\
    \hline
       \end{tabular}
\end{table}

\begin{figure}[ht]
  \centering
  \begin{subfigure}[c]{0.8\textwidth}
  \includegraphics[width=\textwidth,height=5.4cm]{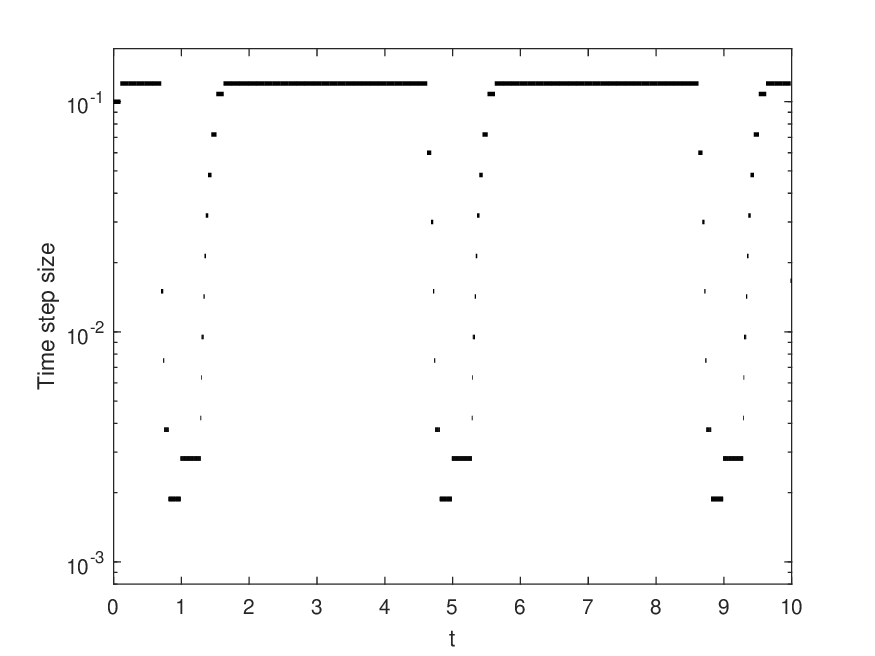}
  \end{subfigure}
      \vspace{0.5cm}
  \begin{subfigure}[c]{0.8\textwidth}
    \includegraphics[width=\textwidth,height=5.4cm]{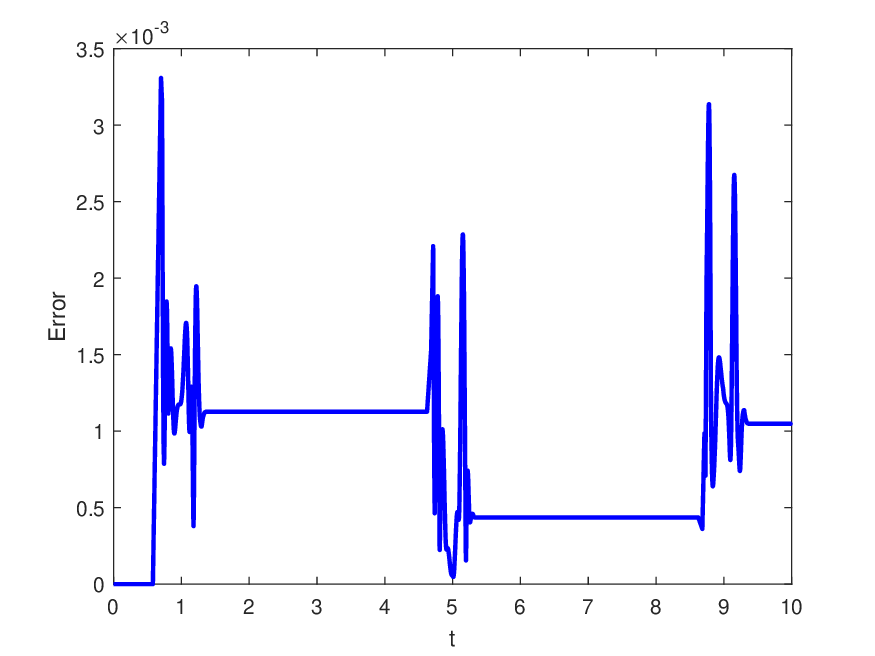}
  \end{subfigure}
    \caption{Example \ref{example4}: Time step-size trajectory (top row) and the error (bottom row) with $Tol=0.1$, $k_0=0.1$, and $k_{\max}=0.12$}\label{fig4.1}
\end{figure}

\section{Conclusions}\label{sec:conclusions}
In this paper, we established the a posteriori error analysis for the IF midpoint method for second order evolution equations. Optimal order residual-based error estimates were derived in the $L^{\infty}$-in-time and energy-in-space norms, using the time reconstruction and the energy technique. An adaptive time-stepping algorithm was developed based on the reliable a posteriori error control.
Finally, numerical results demonstrate convergence rates of the method and the a posteriori quantity, and the high efficiency of the proposed adaptive algorithm.
\section*{Acknowledgments}
We would like to thank the anonymous reviewers for their
insightful comments and suggestions, which helped improve our paper considerably.
\bibliographystyle{siamplain}
\bibliography{references}
\end{document}

%% file: ex_shared.tex
\usepackage{lipsum}
\usepackage{amsfonts}
\usepackage{graphicx}
\usepackage{epstopdf}
\usepackage{algorithmic}
\ifpdf
  \DeclareGraphicsExtensions{.eps,.pdf,.png,.jpg}
\else
  \DeclareGraphicsExtensions{.eps}
\fi

\newsiamremark{remark}{Remark}
\newsiamremark{hypothesis}{Hypothesis}
\crefname{hypothesis}{Hypothesis}{Hypotheses}
\newsiamthm{claim}{Claim}
\newsiamremark{fact}{Fact}
\crefname{fact}{Fact}{Facts}

\headers{An adaptive time-stepping method for linear hyperbolic equations}{Xianfa Hu, Fazhan Geng, and Wansheng Wang}

\title{An adaptive integrating factor midpoint method for second order evolution equations\thanks{Submitted to the editors DATE.
\funding{The first author was partially supported by the Natural Science Foundation of China (Grant Nos. 12271367, 12071419) and  the Natural Science Foundation of Shandong Province, China (Grant No. ZR2024MA056).  The third author is the corresponding author. He was supported by the
Natural Science Foundation of China (Grant No. 12271367), Shanghai Science and Technology Planning Projects (Grant No. 20JC1414200), and Natural Science Foundation of Shanghai (Grant No. 20ZR1441200).}}}

\author{Xianfa Hu\thanks{ School of Mathematics and Statistics, Suzhou University of Technology, Changshu, 215500, People's Republic of China
  (\email{zzxyhxf@163.com}, \email{gengfazhan@sina.com}).} 
   \and Fazhan  Geng\footnotemark[2]
\and Wansheng Wang\thanks{Department of Mathematics, Shanghai Normal University, Shanghai, 200234, People's Republic of China
  (\email{w.s.wang@163.com}).}}

\usepackage{amsopn}


%% file: ex_article.bbl
\begin{thebibliography}{10}

\bibitem{Ainsworth1997}
{\sc M.~Ainsworth and J.~T. Oden}, {\em A posteriori error estimation in finite
  element analysis}, Comput. Methods Appl. Mech. Engrg., 142 (1997), pp.~1--88,
  \url{https://doi.org/10.1016/S0045-7825(96)01107-3}.

\bibitem{Akrivis2010}
{\sc G.~Akrivis and P.~Chatzipantelidis}, {\em A posteriori error estimates for
  the two-step backward differentiation formula method for parabolic
  equations}, SIAM J. Numer. Anal., 48 (2010), pp.~109--132,
  \url{https://doi.org/10.1137/090756995}.

\bibitem{Akrivis2006}
{\sc G.~Akrivis, C.~Makridakis, and R.~H. Nochetto}, {\em A posteriori error
  estimates for the {C}rank-{N}icolson method for parabolic equations}, Math.
  Comp., 75 (2006), pp.~511--531,
  \url{https://doi.org/10.1090/S0025-5718-05-01800-4}.

\bibitem{Akrivis2011}
{\sc G.~Akrivis, C.~Makridakis, and R.~H. Nochetto}, {\em Galerkin and
  {R}unge-{K}utta methods: unified formulation, a posteriori error estimates
  and nodal superconvergence}, Numer. Math., 118 (2011), pp.~429--456,
  \url{https://doi.org/10.1007/s00211-011-0363-6}.

\bibitem{Al-Mohy2011}
{\sc A.~H. Al-Mohy and N.~J. Higham}, {\em Computing the action of the matrix
  exponential, with an application to exponential integrators}, SIAM J. Sci.
  Comput., 33 (2011), pp.~488--511, \url{https://doi.org/10.1137/100788860}.

\bibitem{Babuska1978}
{\sc I.~Babu\v{s}ka and W.~C. Rheinboldt}, {\em Error estimates for adaptive
  finite element computations}, SIAM J. Numer. Anal., 15 (1978), pp.~736--754,
  \url{https://doi.org/10.1137/0715049}.

\bibitem{Bank1985}
{\sc R.~E. Bank and A.~Weiser}, {\em Some a posteriori error estimators for
  elliptic partial differential equations}, Math. Comp., 44 (1985),
  pp.~283--301, \url{https://doi.org/10.2307/2007953}.

\bibitem{Bao2014}
{\sc W.~Bao and Y.~Cai}, {\em Uniform and optimal error estimates of an
  exponential wave integrator sine pseudospectral method for the nonlinear
  {S}chr\"odinger equation with wave operator}, SIAM J. Numer. Anal., 52
  (2014), pp.~1103--1127, \url{https://doi.org/10.1137/120866890}.

\bibitem{Bao2024}
{\sc W.~Bao and C.~Wang}, {\em An explicit and symmetric exponential wave
  integrator for the nonlinear {S}chr\"odinger equation with low regularity
  potential and nonlinearity}, SIAM J. Numer. Anal., 62 (2024), pp.~1901--1928,
  \url{https://doi.org/10.1137/23M1615656}.

\bibitem{Bergermann2024}
{\sc K.~Bergermann and M.~Stoll}, {\em Adaptive rational {K}rylov methods for
  exponential {R}unge-{K}utta integrators}, SIAM J. Matrix Anal. Appl., 45
  (2024), pp.~744--770, \url{https://doi.org/10.1137/23M1559439}.

\bibitem{Berland2005}
{\sc H.~v. Berland, B.~Owren, and B.~r. Skaflestad}, {\em {$B$}-series and
  order conditions for exponential integrators}, SIAM J. Numer. Anal., 43
  (2005), pp.~1715--1727, \url{https://doi.org/10.1137/040612683}.

\bibitem{Calvo2006}
{\sc M.~P. Calvo and C.~Palencia}, {\em A class of explicit multistep
  exponential integrators for semilinear problems}, Numer. Math., 102 (2006),
  pp.~367--381, \url{https://doi.org/10.1007/s00211-005-0627-0}.

\bibitem{Calvo2008}
{\sc M.~P. Calvo and A.~M. Portillo}, {\em Variable step implementation of
  {ETD} methods for semilinear problems}, Appl. Math. Comput., 196 (2008),
  pp.~627--637, \url{https://doi.org/10.1016/j.amc.2007.06.025}.

\bibitem{Cao2025}
{\sc J.~Cao, W.~Wang, and A.~Xiao}, {\em Adaptive fast {$L1-2$} scheme for
  solving time fractional parabolic problems}, Comput. Math. Appl., 179 (2025),
  pp.~59--76, \url{https://doi.org/10.1016/j.camwa.2024.12.003}.

\bibitem{Celledoni2008}
{\sc E.~Celledoni, D.~Cohen, and B.~Owren}, {\em Symmetric exponential
  integrators with an application to the cubic {S}chr\"odinger equation},
  Found. Comput. Math., 8 (2008), pp.~303--317,
  \url{https://doi.org/10.1007/s10208-007-9016-7}.

\bibitem{Celledoni2009}
{\sc E.~Celledoni and B.~K. Kometa}, {\em Semi-{L}agrangian {R}unge-{K}utta
  exponential integrators for convection dominated problems}, J. Sci. Comput.,
  41 (2009), pp.~139--164, \url{https://doi.org/10.1007/s10915-009-9291-3}.

\bibitem{Celledoni2011}
{\sc E.~Celledoni and B.~K. Kometa}, {\em Semi-{L}agrangian multistep
  exponential integrators for index 2 differential-algebraic systems}, J.
  Comput. Phys., 230 (2011), pp.~3413--3429,
  \url{https://doi.org/10.1016/j.jcp.2011.01.036}.

\bibitem{Chen2004}
{\sc Z.~Chen and J.~Feng}, {\em An adaptive finite element algorithm with
  reliable and efficient error control for linear parabolic problems}, Math.
  Comp., 73 (2004), pp.~1167--1193,
  \url{https://doi.org/10.1090/S0025-5718-04-01634-5}.

\bibitem{Cox2002}
{\sc S.~M. Cox and P.~C. Matthews}, {\em Exponential time differencing for
  stiff systems}, J. Comput. Phys., 176 (2002), pp.~430--455,
  \url{https://doi.org/10.1006/jcph.2002.6995}.

\bibitem{Du2021}
{\sc Q.~Du, L.~Ju, X.~Li, and Z.~Qiao}, {\em Maximum bound principles for a
  class of semilinear parabolic equations and exponential time-differencing
  schemes}, SIAM Rev., 63 (2021), pp.~317--359,
  \url{https://doi.org/10.1137/19M1243750}.

\bibitem{Eriksson1995}
{\sc K.~Eriksson, D.~Estep, P.~Hansbo, and C.~Johnson}, {\em Introduction to
  adaptive methods for differential equations}, in Acta numerica, 1995, Acta
  Numer., Cambridge Univ. Press, Cambridge, 1995, pp.~105--158,
  \url{https://doi.org/10.1017/S0962492900002531}.

\bibitem{Ern2017}
{\sc A.~Ern, I.~Smears, and M.~Vohral\'ik}, {\em Guaranteed, locally space-time
  efficient, and polynomial-degree robust a~posteriori error estimates for
  high-order discretizations of parabolic problems}, SIAM J. Numer. Anal., 55
  (2017), pp.~2811--2834, \url{https://doi.org/10.1137/16M1097626}.

\bibitem{Ern2015}
{\sc A.~Ern and M.~Vohral\'ik}, {\em Polynomial-degree-robust a posteriori
  estimates in a unified setting for conforming, nonconforming, discontinuous
  {G}alerkin, and mixed discretizations}, SIAM J. Numer. Anal., 53 (2015),
  pp.~1058--1081, \url{https://doi.org/10.1137/130950100}.

\bibitem{Evans1998}
{\sc L.~C. Evans}, {\em Partial differential equations}, vol.~19 of Graduate
  Studies in Mathematics, American Mathematical Society, Providence, RI, 1998,
  \url{https://doi.org/10.1090/gsm/019}.

\bibitem{Feng2024}
{\sc Y.~Feng and K.~Schratz}, {\em Improved uniform error bounds on a
  {L}awson-type exponential integrator for the long-time dynamics of
  sine-{G}ordon equation}, Numer. Math., 156 (2024), pp.~1455--1477,
  \url{https://doi.org/10.1007/s00211-024-01423-w}.

\bibitem{Gaudreault2018}
{\sc S.~Gaudreault, G.~Rainwater, and M.~Tokman}, {\em K{IOPS}: a fast adaptive
  {K}rylov subspace solver for exponential integrators}, J. Comput. Phys., 372
  (2018), pp.~236--255, \url{https://doi.org/10.1016/j.jcp.2018.06.026}.

\bibitem{Georgoulis2024}
{\sc E.~H. Georgoulis, E.~J.~C. Hall, and C.~G. Makridakis}, {\em On a
  posteriori error estimation for {R}unge-{K}utta discontinuous {G}alerkin
  methods for linear hyperbolic problems}, Stud. Appl. Math., 153 (2024),
  12772, \url{https://doi.org/10.1111/sapm.12772}.

\bibitem{Georgoulis2016}
{\sc E.~H. Georgoulis, O.~Lakkis, C.~G. Makridakis, and J.~M. Virtanen}, {\em A
  posteriori error estimates for leap-frog and cosine methods for second order
  evolution problems}, SIAM J. Numer. Anal., 54 (2016), pp.~120--136,
  \url{https://doi.org/10.1137/140996318}.

\bibitem{Higham2008}
{\sc N.~J. Higham}, {\em Functions of matrices}, Society for Industrial and
  Applied Mathematics (SIAM), Philadelphia, PA, 2008,
  \url{https://doi.org/10.1137/1.9780898717778}.

\bibitem{Hochbruck2020}
{\sc M.~Hochbruck, J.~Leibold, and A.~Ostermann}, {\em On the convergence of
  {L}awson methods for semilinear stiff problems}, Numer. Math., 145 (2020),
  pp.~553--580, \url{https://doi.org/10.1007/s00211-020-01120-4}.

\bibitem{Hochbruck2005}
{\sc M.~Hochbruck and A.~Ostermann}, {\em Explicit exponential {R}unge-{K}utta
  methods for semilinear parabolic problems}, SIAM J. Numer. Anal., 43 (2005),
  pp.~1069--1090, \url{https://doi.org/10.1137/040611434}.

\bibitem{Hochbruck2010}
{\sc M.~Hochbruck and A.~Ostermann}, {\em Exponential integrators}, Acta
  Numer., 19 (2010), pp.~209--286,
  \url{https://doi.org/10.1017/S0962492910000048}.

\bibitem{Hu2025b}
{\sc X.~Hu, W.~Wang, and Y.~Fang}, {\em A {P}osteriori {E}rror {E}stimates for
  {E}xponential {M}idpoint {I}ntegrator {F}inite {E}lement {M}ethod for
  {P}arabolic {E}quations}, Math. Methods Appl. Sci., 48 (2025),
  pp.~5849--5862, \url{https://doi.org/10.1002/mma.10641}.

\bibitem{Hu2025a}
{\sc X.~Hu, W.~Wang, M.~Mao, and J.~Cao}, {\em A posteriori error estimates for
  the exponential midpoint method for linear and semilinear parabolic
  equations}, Numer. Algorithms, 99 (2025), pp.~1983--2010,
  \url{https://doi.org/10.1007/s11075-024-01940-7}.

\bibitem{Huang2013}
{\sc J.~Huang, J.~Lai, and T.~Tang}, {\em An adaptive time stepping method with
  efficient error control for second-order evolution problems}, Sci. China
  Math., 56 (2013), pp.~2753--2771,
  \url{https://doi.org/10.1007/s11425-013-4730-x}.

\bibitem{Isherwood2018}
{\sc L.~Isherwood, Z.~J. Grant, and S.~Gottlieb}, {\em Strong stability
  preserving integrating factor {R}unge-{K}utta methods}, SIAM J. Numer. Anal.,
  56 (2018), pp.~3276--3307, \url{https://doi.org/10.1137/17M1143290}.

\bibitem{Ju2021}
{\sc L.~Ju, X.~Li, Z.~Qiao, and J.~Yang}, {\em Maximum bound principle
  preserving integrating factor {R}unge-{K}utta methods for semilinear
  parabolic equations}, J. Comput. Phys., 439 (2021), 110405,
  \url{https://doi.org/10.1016/j.jcp.2021.110405}.

\bibitem{Kassam2005}
{\sc A.-K. Kassam and L.~N. Trefethen}, {\em Fourth-order time-stepping for
  stiff {PDE}s}, SIAM J. Sci. Comput., 26 (2005), pp.~1214--1233,
  \url{https://doi.org/10.1137/S1064827502410633}.

\bibitem{Krogstad2005}
{\sc S.~Krogstad}, {\em Generalized integrating factor methods for stiff
  {PDE}s}, J. Comput. Phys., 203 (2005), pp.~72--88,
  \url{https://doi.org/10.1016/j.jcp.2004.08.006}.

\bibitem{Lakkis2006}
{\sc O.~Lakkis and C.~Makridakis}, {\em Elliptic reconstruction and a
  posteriori error estimates for fully discrete linear parabolic problems},
  Math. Comp., 75 (2006), pp.~1627--1658,
  \url{https://doi.org/10.1090/S0025-5718-06-01858-8}.

\bibitem{Lawson1967}
{\sc J.~D. Lawson}, {\em Generalized {R}unge-{K}utta processes for stable
  systems with large {L}ipschitz constants}, SIAM J. Numer. Anal., 4 (1967),
  pp.~372--380, \url{https://doi.org/10.1137/0704033}.

\bibitem{Li2020}
{\sc B.~Li, J.~Yang, and Z.~Zhou}, {\em Arbitrarily high-order exponential
  cut-off methods for preserving maximum principle of parabolic equations},
  SIAM J. Sci. Comput., 42 (2020), pp.~A3957--A3978,
  \url{https://doi.org/10.1137/20M1333456}.

\bibitem{Li2021}
{\sc J.~Li, X.~Li, L.~Ju, and X.~Feng}, {\em Stabilized integrating factor
  {R}unge-{K}utta method and unconditional preservation of maximum bound
  principle}, SIAM J. Sci. Comput., 43 (2021), pp.~A1780--A1802,
  \url{https://doi.org/10.1137/20M1340678}.

\bibitem{Li2016}
{\sc Y.-W. Li and X.~Wu}, {\em Exponential integrators preserving first
  integrals or {L}yapunov functions for conservative or dissipative systems},
  SIAM J. Sci. Comput., 38 (2016), pp.~A1876--A1895,
  \url{https://doi.org/10.1137/15M1023257}.

\bibitem{Makridakis2003}
{\sc C.~Makridakis and R.~H. Nochetto}, {\em Elliptic reconstruction and a
  posteriori error estimates for parabolic problems}, SIAM J. Numer. Anal., 41
  (2003), pp.~1585--1594, \url{https://doi.org/10.1137/S0036142902406314}.

\bibitem{Mei2017}
{\sc L.~Mei and X.~Wu}, {\em Symplectic exponential {R}unge-{K}utta methods for
  solving nonlinear {H}amiltonian systems}, J. Comput. Phys., 338 (2017),
  pp.~567--584, \url{https://doi.org/10.1016/j.jcp.2017.03.018}.

\bibitem{Moler2003}
{\sc C.~Moler and C.~Van~Loan}, {\em Nineteen dubious ways to compute the
  exponential of a matrix, twenty-five years later}, SIAM Rev., 45 (2003),
  pp.~3--49, \url{https://doi.org/10.1137/S00361445024180}.

\bibitem{Niesen2012}
{\sc J.~Niesen and W.~M. Wright}, {\em Algorithm 919: a {K}rylov subspace
  algorithm for evaluating the {$\phi$}-functions appearing in exponential
  integrators}, ACM Trans. Math. Software, 38 (2012), 22,
  \url{https://doi.org/10.1145/2168773.2168781}.

\bibitem{Nochetto2000}
{\sc R.~H. Nochetto, G.~Savar\'e, and C.~Verdi}, {\em A posteriori error
  estimates for variable time-step discretizations of nonlinear evolution
  equations}, Comm. Pure Appl. Math., 53 (2000), pp.~525--589,
  \url{https://doi.org/10.1002/(SICI)1097-0312(200005)53:5<525::AID-CPA1>3.0.CO;2-M}.

\bibitem{Ostermann2018}
{\sc A.~Ostermann and K.~Schratz}, {\em Low regularity exponential-type
  integrators for semilinear {S}chr\"odinger equations}, Found. Comput. Math.,
  18 (2018), pp.~731--755, \url{https://doi.org/10.1007/s10208-017-9352-1}.

\bibitem{Ostermann2020}
{\sc A.~Ostermann and C.~Su}, {\em A {L}awson-type exponential integrator for
  the {K}orteweg--de {V}ries equation}, IMA J. Numer. Anal., 40 (2020),
  pp.~2399--2414, \url{https://doi.org/10.1093/imanum/drz030}.

\bibitem{Verfurth1994}
{\sc R.~Verf\"urth}, {\em A posteriori error estimation and adaptive
  mesh-refinement techniques}, in Proceedings of the {F}ifth {I}nternational
  {C}ongress on {C}omputational and {A}pplied {M}athematics ({L}euven, 1992),
  vol.~50, 1994, pp.~67--83,
  \url{https://doi.org/10.1016/0377-0427(94)90290-9}.

\bibitem{Wang2016}
{\sc B.~Wang, A.~Iserles, and X.~Wu}, {\em Arbitrary-order trigonometric
  {F}ourier collocation methods for multi-frequency oscillatory systems},
  Found. Comput. Math., 16 (2016), pp.~151--181,
  \url{https://doi.org/10.1007/s10208-014-9241-9}.

\bibitem{Wang2023}
{\sc B.~Wang and X.~Zhao}, {\em Geometric two-scale integrators for highly
  oscillatory system: uniform accuracy and near conservations}, SIAM J. Numer.
  Anal., 61 (2023), pp.~1246--1277, \url{https://doi.org/10.1137/21M1462908}.

\bibitem{Wang2022a}
{\sc W.~Wang, M.~Mao, and Y.~Huang}, {\em Optimal a posteriori estimators for
  the variable step-size {BDF}2 method for linear parabolic equations}, J.
  Comput. Appl. Math., 413 (2022), 114306,
  \url{https://doi.org/10.1016/j.cam.2022.114306}.

\bibitem{Wang2021}
{\sc W.~Wang, M.~Mao, and Z.~Wang}, {\em Stability and error estimates for the
  variable step-size {BDF}2 method for linear and semilinear parabolic
  equations}, Adv. Comput. Math., 47 (2021), 8,
  \url{https://doi.org/10.1007/s10444-020-09839-2}.

\bibitem{Wang2018}
{\sc W.~Wang, T.~Rao, W.~Shen, and P.~Zhong}, {\em A posteriori error analysis
  for {C}rank-{N}icolson-{G}alerkin type methods for reaction-diffusion
  equations with delay}, SIAM J. Sci. Comput., 40 (2018), pp.~A1095--A1120,
  \url{https://doi.org/10.1137/17M1143514}.

\bibitem{Wang2022b}
{\sc W.~Wang and L.~Yi}, {\em Delay-dependent elliptic reconstruction and
  optimal {$L^\infty(L^2)$} a posteriori error estimates for fully discrete
  delay parabolic problems}, Math. Comp., 91 (2022), 4473098, pp.~2609--2643,
  \url{https://doi.org/10.1090/mcom/3761}.

\bibitem{Wang2020}
{\sc W.~Wang, L.~Yi, and A.~Xiao}, {\em A posteriori error estimates for fully
  discrete finite element method for generalized diffusion equation with
  delay}, J. Sci. Comput., 84 (2020), 13,
  \url{https://doi.org/10.1007/s10915-020-01262-5}.

\end{thebibliography}
